\documentclass[a4j,12pt,final]{article}
\usepackage{amsmath}
\usepackage{amssymb}
\usepackage{amsthm}
\usepackage{amscd}
\newtheorem{theorem}{Theorem}[section]

\theoremstyle{definition}
\newtheorem{definition}{Definition}[section]
\theoremstyle{remark}
\newtheorem{remark}{Remark}[section]
\theoremstyle{proposition}
\newtheorem{proposition}{Proposition}[section]
\numberwithin{equation}{section}
\newtheorem{example}{Example}[section]
\theoremstyle{corollary}
\newtheorem{corollary}{Corollary}[section]

\theoremstyle{conjecture}

\begin{document}

\title{The Chern Character in the Simplicial \\
de Rham Complex}
\author{Naoya Suzuki}
\date{}
\maketitle
\begin{abstract}
On the basis of Dupont's work, we exhibit a cocycle in the simplicial de Rham complex
which represents the Chern character. We also prove the related conjecture due to Brylinski. This gives a way to construct a cocycle
in a local truncated complex.
\end{abstract}

\section{Introduction}
\setcounter{equation}{0}
%

It is well-known that there is one-to-one correspondence between the characteristic classes of $G$-bundles and the elements 
in the cohomology ring of the classifying space $BG$.  So it is important to investigate $H^*(BG)$ in research on the
characteristic classes. However, in general $BG$ is not a manifold so we can not
adapt the usual de Rham theory on it. To overcome this problem, a total complex of a double complex ${\Omega}^{*} (NG(*)) $ which is associated to a simplicial manifold $\{ NG(*) \}$ is often used. 
In brief, $\{ NG(*) \}$  is a sequence of manifolds $\{ NG(p) =G^p \}_{ p=0,1, \cdots }$ together
with face operators ${\varepsilon}_{i} : NG(p) \rightarrow NG(p-1)  $ for $ i= 0, \cdots , p $ satisfying relations ${\varepsilon}_{i}{\varepsilon}_{j} ={\varepsilon}_{j-1}{\varepsilon}_{i}$ for $i<j$ (The standard definition also involves 
degeneracy operators but we do not need them here). The cohomology ring of
${\Omega}^{*} (NG(*)) $ is isomorphic to $H^*(BG)$ so we can use this complex as a candidate of the de Rham complex on $BG$.

In  \cite{Dup}, Dupont introduced another double complex $A^{*,*}(NG)$ on $NG$ and showed the cohomology ring of its total complex $A^*(NG)$
is also isomorphic to $H^*(BG)$. Then he used it to construct a homomorphism from $I^* (G)$, the $G$-invariant polynomial ring over Lie algebra $\mathcal{G}$, to $H^*(BG)$ for a classical Lie group $G$.   

The images of this homomorphism in ${\Omega}^{*} (NG(*)) $ are called the Bott-Shulman-Stasheff forms.
The main purpose of this paper is to exhibit these cocycles precisely when they represent the Chern characters.

In addition, we also show that the conjecture due to Brylinski in \cite{Bry} is true. This gives a way to construct a cocycle
in a local truncated complex $[\sigma_{<p}\Omega^{*} _{\rm loc}(NG)]$ whose cohomology class is mapped to the cohomology class of the Bott-Shulman-Stasheff form in a local double complex by a boundary map. His original motivation to introduce these complexes 
and the conjecture is to study the local cohomology group of the gauge group ${\rm Map}(X,G)$ and the Lie algebra cohomology of 
its Lie algebra. Actually, as a special case $X=S^1$, 
he constructed the standard Kac-Moody $2$-cocycle for a loop Lie algebra by using the cocycle in the local truncated complex $[\sigma_{<2}\Omega^{3} _{\rm loc}(NG)]$.

The outline of this paper is as follows. In section 2, we briefly recall the universal Chern-Weil theory due to Dupont.
In section 3, we obtain the  Bott-Shulman-Stasheff form in ${\Omega}^{*} (NG(*)) $ which represents the Chern character ${\rm ch}_p $.
In section 4, we introduce some result about the Chern-Simons forms. In section 5, we prove Brylinski's conjecture.

\section{Review of the universal Chern-Weil Theory}
\setcounter{equation}{0}
In this section we recall the universal Chern-Weil theory following \cite{Dup2}.
For any Lie group $G$, we have simplicial manifolds $NG$, $N \bar{G}$ and simplicial $G$-bundle  $\gamma : N \bar{G} \rightarrow NG$
as follows:\\
\par
$NG(q)  = \overbrace{G \times \cdots \times G }^{q-times}  \ni (h_1 , \cdots , h_q ) :$  \\
face operators \enspace ${\varepsilon}_{i} : NG(q) \rightarrow NG(q-1)  $
$$
{\varepsilon}_{i}(h_1 , \cdots , h_q )=\begin{cases}
(h_2 , \cdots , h_q )  &  i=0 \\
(h_1 , \cdots ,h_i h_{i+1} , \cdots , h_q )  &  i=1 , \cdots , q-1 \\
(h_1 , \cdots , h_{q-1} )  &  i=q.
\end{cases}
$$

\par
\medskip
$N \bar{G} (q) = \overbrace{ G \times \cdots \times G }^{q+1 - times} \ni (g_1 , \cdots , g_{q+1} ) :$ \\
face operators \enspace $ \bar{\varepsilon}_{i} : N \bar{G}(q) \rightarrow N \bar{G}(q-1)  $ 
$$ \bar{{\varepsilon}} _{i} (g_0 , \cdots , g_{q} ) = (g_0 , \cdots , g_{i-1} , g_{i+1}, \cdots , g_{q})  \qquad i=0,1, \cdots ,q. $$

\par
\medskip

We define $\gamma : N \bar{G} \rightarrow NG $ as $ \gamma (g_0 , \cdots , g_{q} ) = (g_0 {g_1}^{-1} , \cdots , g_{q-1} {g_{q}}^{-1} )$.\\
\par
For any simplicial manifold $X = \{ X_* \}$, we can associate a topological space $\parallel X \parallel $ 
called the fat realization. 
Since any $G$-bundle $\pi : E \rightarrow M$ can be realized as a pull-back of the fat realization of $\gamma $,
$\parallel \gamma \parallel$ is the universal bundle $EG \rightarrow BG$  \cite{Seg}. 

Now we construct a double complex associated to a simplicial manifold.

\begin{definition}
For any simplicial manifold $ \{ X_* \}$ with face operators $\{ {\varepsilon}_* \}$, we define a double complex as follows:
$${\Omega}^{p,q} (X) := {\Omega}^{q} (X_p) $$
Derivatives are:
$$ d' := \sum _{i=0} ^{p+1} (-1)^{i} {\varepsilon}_{i} ^{*}  , \qquad  d'' := (-1)^{p} \times {\rm the \enspace exterior \enspace differential \enspace on \enspace }{ \Omega ^*(X_p) }. $$
\end{definition}
\bigskip

For $NG$ and $N \bar{G} $ the following holds \cite{Bot2} \cite{Dup2} \cite{Mos}.

\begin{theorem}
There exist ring isomorphisms 
$$ H({\Omega}^{*} (NG))  \cong  H^{*} (BG ), \qquad  H({\Omega}^{*} (N \bar{G})) \cong H^{*} (EG ). $$
Here ${\Omega}^{*} (NG)$  and  ${\Omega}^{*} (N \bar{G})$  mean the total complexes.
\end{theorem} 
For example, the derivative $d'+d'':\Omega ^{p}(NG) \rightarrow \Omega ^{p+1}(NG)$ is 
given as follows:

$$
\begin{CD}
{\Omega}^{p} (G ) \\
@AA{-d}A \\
{\Omega}^{p-1} (G )@>{{\varepsilon}_{0} ^{*} - {\varepsilon}_{1} ^{*} +{\varepsilon}_{2} ^{*} }>>{\Omega}^{p-1} (NG(2))\\
@.@AAdA\\
@.{\Omega}^{p-2} (NG(2))\\
@.@. \ddots \\
@.@.@.{\Omega}^{1} (NG(p)) \\
@.@.@.@AA{(-1)^p d }A\\
@.@.@.{\Omega}^{0} (NG(p))@>{ \sum _{i=0} ^{p+1} (-1)^{i} {\varepsilon}_{i} ^{*}}>> {\Omega}^{0} (NG(p+1)) 
\end{CD}
$$

\bigskip
\begin{remark}
Let $\pi : P \rightarrow M$ be a principal $G$-bundle and $\{ g_{\alpha \beta}:U_{\alpha \beta} \rightarrow G \}$ be 
the transition functions of it. Then we can pull-back the cocycle in $\Omega ^*(NG)$ to the \v{C}ech-de Rham complex of 
$M$ by $\{g_{\alpha \beta} \}$. When $\kappa$ is the characteristic class which corresponds to the  cocycle in $\Omega ^*(NG)$,
the image of $ g_{\alpha \beta} ^*$ in $H^* _{\check{C}ech-de Rham} (M)$ is the characteristic class $\kappa (P)$ of $\pi : P \rightarrow M$. 
For more details, see for instance \cite{Mos}.
\end{remark}

\bigskip

There is another double complex associated to a simplicial manifold.
\begin{definition}[\cite{Dup}]
A simplicial $n$-form on a simplicial manifold $ \{ {X}_{p} \} $ is a sequence $ \{ {\phi}^{(p)} \}$
of $n$-forms ${\phi}^{(p)}$ on ${\Delta}^{p} \times {X}_{p} $ such that
$${({\varepsilon}^{i} \times id )}^{*} {\phi}^{(p)} = {(id \times {\varepsilon}_{i} )}^{*} {\phi}^{(p-1)}. $$
Here ${\varepsilon}^{i}$ is the canonical $i$-th face operator of ${\Delta}^{p}$.\\
\end{definition}

{\rm Let}  \thinspace $A^{k,l} (X)$ be the set of all simplicial $(k+l)$-forms on ${\Delta}^{p} \times {X}_{p} $ which are expressed locally 
of the form
$$ \sum { a_{ i_1 \cdots i_k j_1 \cdots j_l } (dt_{i_1 } \wedge \cdots \wedge dt_{i_k } \wedge dx_{j_1 } \wedge \cdots \wedge dx_{j_l })}$$
where $(t_0, t_1, \cdots, t_p)$ are the barycentric coordinates in ${\Delta}^{p}  $ and $x_j $ are the local coordinates in $ {X}_{p} $.
We call these forms $(k,l)$-form on ${\Delta}^{p} \times {X}_{p} $ and define derivatives as:
$$ d' := {\rm  the \enspace exterior \enspace differential \enspace on \enspace } {\Delta}^{p}  $$
$$ d'' := (-1)^{k} \times {\rm  the \enspace exterior \enspace differential \enspace on \enspace } {X_p }. $$
Then $(A^{k,l} (X) , d' , d'' )$ is a double complex.\\

Let $A^{*} (X)$ denote the total complex of $A^{*,*}(X)$. We define a map $I_{\Delta} : A^{*} (X)  \rightarrow {\Omega}^{*} (X) $ as follows: 
$$ I_{ \Delta }( \alpha ) :=  \int_{{\Delta }^{p}} ( { \alpha } |_{{ \Delta }^{p} \times {X}_{p} } ). $$
Then the following theorem holds \cite{Dup}.

\begin{theorem}
$ I_{ \Delta } $ { induces a natural ring isomorphism}
$$ I_{ \Delta } ^{*} : H( A^{*} (X)) \cong H({\Omega}^{*} (X)). $$

\end{theorem}
\bigskip

Let  $\mathcal{G}$ denote the Lie algebra of $G$. A connection on a simplicial $G$-bundle $\pi : \{ E_p \} \rightarrow \{ M_p \} $ is a sequence of $1$-forms $\{ \theta \}$ on $\{ E_p \}$ with coefficients $\mathcal{G}$
such that $\theta $ restricted  to ${\Delta}^{p} \times {E}_{p} $ is a usual connection form on a principal $G$-bundle  ${\Delta}^{p} \times {E}_{p} \rightarrow {\Delta}^{p} \times {M}_{p} $.

Dupont constructed a canonical connection $\theta \in A^1 (N \bar{G} )$ on ${\gamma} : N \bar{G} \rightarrow NG $ in the following way:

$$ {\theta } |_{{ \Delta }^{p} \times N \bar{G} (p)} := t_0 {\theta }_0 + \cdots + t_{p} {\theta }_{p}. $$

Here ${\theta }_i $ is defined by ${\theta }_i = {\rm pr}_i ^{*} \bar {\theta } $ where ${\rm pr}_i : { \Delta }^{p} \times N \bar{G} (p) \rightarrow G $ is the projection into the $i$-th factor of $ N \bar{G} (p) $ and $\bar {\theta }$ is the Maurer-Cartan form of $G$.
We also obtain its curvature $\Omega \in A^2 (N \bar{G} )$ on ${\gamma }$ as:
$$ \Omega |_{{ \Delta }^{p} \times N \bar{G} (p) }= d \theta |_{{ \Delta }^{p} \times N \bar{G} (p) }
+  \frac{1}{2} [ \theta |_{{ \Delta }^{p} \times N \bar{G} (p) } , \theta |_{{ \Delta }^{p} \times N \bar{G} (p) } ]. $$

Let  ${\rm  I}^{*} (G)$ denote the ring of $G$-invariant polynomials on $\mathcal{G}$. For $P \in I^* (G)$, we restrict 
$P( \Omega ) \in A^{*} (N \bar{G} )$ to each ${\Delta}^{p} \times N \bar{G} (p) 
\rightarrow {\Delta}^{p} \times NG(p) $ and apply the usual Chern-Weil theory then we have a simplicial $2k$-form  $P( \Omega )$
on $NG$. 

Now we have a canonical homomorphism 
$${w}:{\rm  I}^{*} (G) \rightarrow H({\Omega}^{*} (NG)) $$
which maps $P \in I^* (G)$  to ${w}(P)=[I_{\Delta } ( P({\Omega}) )]$.

\section{The Chern character in the double complex}
\setcounter{equation}{0}
In this section we exhibit a cocycle in $\Omega ^{*,*}(NG) $  which represents the
Chern character.
Throughout this section, $G= GL(n ;  \mathbb{C} )$ and $ {\rm ch }_p$ means the $p$-th Chern character.

Note that the diagram below is commutative, since $I_{\Delta }$ acts only on the differential forms on ${ \Delta }^{*} $,
and so does ${\gamma}^{*}$ on differential forms on each $N G (*)$.
$$
\begin{CD}
A^{*,*}(N \bar{G} )@>{I_{\Delta }}>>{\Omega}^{*,*}(N \bar{G} )\\
@A{\gamma}^{*}AA@AA{\gamma}^{*}A\\
A^{*,*}(N G )@>{I_{\Delta }}>>{\Omega}^{*,*}(N G )
\end{CD}
$$

We first give the cocycle in ${\Omega}^{p+q}(N  \bar{G}(p-q)) (0 \leq q \leq p-1)$ which corresponds to the $p$-th Chern character 
by restricting $ ({1}/{p! } )\thinspace {\rm tr} \left(  \left( { - \Omega }/{2 \pi i } \right) ^p  \right) 
\in A^{2p}(N \bar{G} ) $ to $A^{p-q,p+q} ( {\Delta}^{p-q} \times N \bar{G}(p-q))$ and integrating it along ${\Delta}^{p-q}$.
Then we give the cocycle in ${\Omega}^{p+q}(N  {G}(p-q))$ 
which hits to it by ${\gamma}^{*}$.

Since $[\theta _i , \theta _j ]= \theta _i \wedge \theta _j + \theta _j \wedge \theta _i$ for any $i,j$, 
 $$  \Omega |_{{ \Delta }^{p-q} \times N \bar{G} (p-q) }
 = -\sum _{i=1} ^{p-q}  dt_i \wedge (\theta _0 - \theta _{i} )
-  \sum _{ 0 \leq i < j \leq p-q } t_i t_j ( \theta _i - \theta _j  ) ^{2}.$$

Now
$$ dt_i \wedge (\theta _0 -\theta _{i} ) = dt_i \wedge \{ (\theta _0 -\theta _{1} ) + (\theta _{1} -\theta _{2} ) +
\cdots + (\theta _{i-1} -\theta _{i} ) \} $$

and for any ${\mathcal{G} }$-valued differential forms $ \alpha , \beta , \gamma $ and any integer $ 0 \leq \forall x \leq p-q-1 $, the equation
$  \alpha \wedge (dt_i \wedge ( \theta _x - \theta _{x+1} )) \wedge \beta \wedge (dt_j \wedge ( \theta _x - \theta _{x+1} )) \wedge \gamma  = - \alpha \wedge (dt_j \wedge ( \theta _x - \theta _{x+1} )) \wedge \beta \wedge (dt_i \wedge ( \theta _x - \theta _{x+1} )) \wedge \gamma $ holds,
so the terms of the forms above cancel with each other in $ \left( - \Omega |_{{ \Delta }^{p-q} \times N \bar{G} (p-q) } \right) ^{p} $.
Then we see:

$$ \left( - \Omega |_{{ \Delta }^{p-q} \times N \bar{G} (p-q) } \right) ^{p}
 = \left( \sum _{i=1} ^{p-q}  dt_i \wedge (\theta _{i-1} - \theta _{i} )
+ \sum _{ 0 \leq i < j \leq p-q } t_i t_j ( \theta _i - \theta _j  ) ^{2} \right) ^{p}. $$

Now we obtain the following theorem.
\begin{theorem}
{We set:}
$$ \bar{S}_{p-q}=\sum _{\sigma \in \mathfrak{S} _{p-q-1 } }({\rm sgn} ( \sigma ) ) (\theta _{ \sigma (1)}  - \theta _{ \sigma (1)+1 } )
\cdots (\theta _{ \sigma (p-q-1) }  - \theta _{ \sigma (p-q-1)+1 } )$$
{Then the cocycle in } $ {\Omega}^{p+q} ( N \bar{G} (p-q) ) \ (0 \leq q \leq p-1) $ { which corresponds to the $p$-th Chern character
${\rm ch}_p$ is}

$$ \frac{1}{p! } \left(\frac{1}{2 \pi i } \right)^p (-1)^{(p-q)(p-q-1)/2  } \times \hspace{18em} $$
$$    \mathrm{tr} \sum \left( (p(\theta _0 - \theta _{1})) \wedge \bar{H}_q (\bar{S}_{p-q}) \times  \int _{{\Delta}^{p-q}} \prod_{i<j} (t_i t_j)^{a_{ij}(\bar{H}_q (\bar{S}_{p-q}))} dt_1 \wedge \cdots \wedge dt_{p-q} \right).$$
Here $\bar{H}_q (\bar{S}_{p-q})$ means the terms that $ (\theta _{i} - \theta _{j} )^2 \enspace (1 \leq i < j \leq p-q+1 ) $ { are
put }$q${ -times between }$(\theta _{k-1} - \theta _{k} )$ { and }$(\theta _{l} - \theta _{l+1} )$ { in $ \bar{S}_{p-q}$ permitting overlaps;
 $a_{ij}(\bar{H}_q (\bar{S}_{p-q}))$
means the number of }$(\theta _{i} - \theta _{j} )^2 ${ in it.
$\sum$ means the sum of all such terms}.

\end{theorem}

\begin{proof}
The  cocycle in  $ {\Omega}^{p+q} ( N \bar{G} (p-q) ) $  which corresponds to ${\rm ch}_p$ is given by
$$\int_{{\Delta}^{p-q}} \frac{1}{p!} {\rm tr} \left( \left( \frac{- \Omega |_{{ \Delta }^{p-q} \times N \bar{G} (p-q) } }{2 \pi i} \right) ^p \right) \hspace{15em}$$
$$= \frac{1}{p!}  \left(\frac{1}{2 \pi i } \right)^p \int_{{\Delta}^{p-q}}  {\rm tr} \left( \left( \sum _{i=1} ^{p-q}  dt_i \wedge (\theta _{i-1} - \theta _{i} )
+ \sum _{ 0 \leq i < j \leq p-q } t_i t_j ( \theta _i - \theta _j  ) ^{2} \right) ^{p} \right).$$
By calculating this equation, we can check that the statement of Theorem 3.1 is true.
\end{proof}

For the purpose of getting the differential forms in $\Omega ^{*,*}(NG) $ which hit the cocycles in 
Theorem 3.1 by ${\gamma}^{*}$, we set 
$$\varphi_s:=h_1 \cdots h_{s-1}dh_s h^{-1} _s \cdots h^{-1} _1.$$
Here $h_i$ is the $i$-th factor of $NG(*)$.

A straightforward calculation shows that
$${\gamma}^{*} {\rm{tr}} (\varphi_{i_1} \varphi_{i_2}   \cdots \varphi_{i_{p-1}} \varphi_{i_p}  )  =\mathrm{tr}(\theta _{i_1-1} - \theta _{i_1 } ) (\theta _{i_2-1} - \theta _{i_2 } ) \cdots (\theta _{i_p-1} - \theta _{i_p } ).$$

From the above, we conclude:
\begin{theorem}
We set:
$$ R_{ij} = (\varphi_{i}+ \varphi_{i+1} + \cdots + \varphi_{j-1})^{2} \qquad (1 \leq i<j \leq p-q+1 )$$
$$ S_{p-q}   =\sum _{\sigma \in \mathfrak{S} _{p-q-1 } } {\rm sgn} ( \sigma ) \varphi_{ \sigma (1)+1 } \cdots  \varphi_{ \sigma (p-q-1)+1 }.
$$
Then the cocycle in $ \Omega ^{p+q} (NG(p-q)) \ (0 \leq q \leq p-1)$  which represents the $p$-th  Chern character ${\rm ch}_p $ is
$$  \frac{1}{(p-1)! } \left( \frac{1}{2 \pi i } \right)^p (-1)^{(p-q)(p-q-1)/2   } \times  \hspace{17em}$$
$$   \mathrm{tr} \sum \left( \varphi_1 \wedge {H_q} ({S}_{p-q})
 \times \int _{{\Delta}^{p-q}} \prod_{i<j} (t_{i-1} t_{j-1})^{a_{ij}({H_q}({S}))} dt_1 \wedge \cdots \wedge dt_{p-q} \right). $$
Here ${H_q} ({S}_{p-q})$ means the term that 
 $ R_{ij} \enspace (1 \leq i < j \leq p-q+1 ) $  are put $q$ -times between $\psi_k$  and $\psi_l $
 in 
${S}_{p-q} $ permitting overlaps; $a_{ij}({H}_q({S}_{p-q}))$  means the number of 
$R_{ij} $  in it. $\sum$ means the sum of all such terms.
\end{theorem}
\begin{proof}
We can easily check that the cocycle in Theorem 3.2 is mapped to the cochain in Theorem 3.1 by $\gamma^* :  \Omega ^{p+q} (NG(p-q)) 
\rightarrow  \Omega ^{p+q} (N \bar{G}(p-q)) $. The statement folllows from this.
\end{proof}
\begin{remark}
The coefficients in theorem 3.2 are calculated using the following famous formula.
$$ \int _{{\Delta}^{r}}t_{0} ^{b_0} t_1 ^{b_1} \cdots t_{r} ^{b_{r}}dt_1 \wedge \cdots \wedge dt_{r} =
\frac{b_0 ! ~ b_1 ! \cdots b_r !}{(b_0 + b_1 + \cdots + b_r +r)!}.
$$
\end{remark}

\begin{corollary}
The cochain $\omega_{p}$ in $ \Omega ^{2p-1} (NG(1)) $ which corresponds to the $p$-th Chern character is given as follows:
$$\omega_{1}= \frac{1}{p! } \left( \frac{1}{2 \pi i } \right)^p \frac{1}{ {}_{2p-1}C_{p-1}}{\rm tr }(h^{-1} dh )^{2p-1}. $$
\end{corollary}
\bigskip

\begin{corollary}
The cochain $\omega_{p}$ in $ \Omega ^{p} (NG(p)) $ which corresponds to the $p$-th Chern character is given as follows:
$$\omega_{p}= (-1)^{p(p-1)/2   } \frac{1}{p!(p-1)! } \left( \frac{1}{2 \pi i } \right)^p \mathrm{tr}  \left( \varphi_1 \wedge  \sum _{\sigma \in \mathfrak{S} _{p-1 } } {\rm sgn} ( \sigma ) \varphi_{ \sigma (1)+1 } \cdots  \varphi_{ \sigma (p-1)+1 } \right).$$
\end{corollary}
\bigskip

\begin{example}
The cocycle which represents the second  Chern character ${\rm ch}_2 $ in $ \Omega ^{4} (NG) $ is the sum of the following
$C_{1,3}$ and $C_{2,2}$:
$$
\begin{CD}
0 \\
@AA{d''}A \\
C_{1,3} \in {\Omega}^{3} (G )@>{d'}>>{\Omega}^{3} (NG(2))\\
@.@AA{d''}A\\
@.C_{2,2} \in {\Omega}^{2} (NG(2))@>{d'}>> 0
\end{CD}
$$
$$C_{1,3} = \left( \frac{1}{2 \pi i} \right) ^2  \frac{1}{6}{\rm tr}({h^{-1}dh} )^3 ,
\qquad C_{2,2} = \left( \frac{1}{2 \pi i} \right) ^2  \frac{-1}{2}{\rm tr}( dh_1 dh_2 h_2 ^{-1} h_1 ^{-1} ).$$

\end{example}
\begin{corollary}
The cocycle which represents the second  Chern class ${\rm c}_2 $ in $ \Omega ^{4} (NG) $ is the sum of the following
$c_{1,3}$ and $c_{2,2}$:
$$
\begin{CD}
0 \\
@AA{d''}A \\
c_{1,3} \in {\Omega}^{3} (G )@>{d'}>>{\Omega}^{3} (NG(2))\\
@.@AA{d''}A\\
@.c_{2,2} \in {\Omega}^{2} (NG(2))@>{d'}>> 0
\end{CD}
$$
$$c _{1,3} = \left( \frac{1}{2 \pi i} \right) ^2  \frac{-1}{6}{\rm tr}({h^{-1}dh} )^3 \hspace{17em} $$
$$c _{2,2} = \left( \frac{1}{2 \pi i} \right) ^2  \frac{1}{2}{\rm tr}( dh_1 dh_2  h_2 ^{-1} h_1 ^{-1})- \left( \frac{1}{2 \pi i} \right) ^2  \frac{1}{2}{\rm tr}( h_1 ^{-1} dh_1){\rm tr}( h_2 ^{-1}dh_2  ).$$

\end{corollary}
\begin{example}
The cocycle which represents the $3$rd  Chern character ${\rm ch}_3 $ in $ \Omega ^{6} (NG) $ is the sum of the following
$C_{1,5} , C_{2,4}$ and $C_{3,3}$:
$$
\begin{CD}
0 \\
@AA{d''}A \\
{ C_{1,5} \in {\Omega}^{5} (G )}@>{d'}>>{\Omega}^{5} (NG(2))\\
@.@AA{d''}A\\
@.{ C_{2,4} \in {\Omega}^{4} (NG(2))}@>{d'}>> {\Omega}^{4} (NG(3)) \\
@.@.@AA{d''}A\\
@.@.{ C_{3,3} \in {\Omega}^{3} (NG(3))}@>{d'}>> 0
\end{CD}
$$

$$C_{1,5} = \frac{1}{3!} \left( \frac{1}{2 \pi i} \right) ^3 \frac{1 }{10}{\rm tr}({h^{-1}dh} )^5  \hspace{18em} $$

$$    C_{2,4}= \frac{-1}{3!} \left( \frac{1}{2 \pi i} \right) ^3 (
\frac{1 }{2}{\rm tr}  (   dh_1 {h_1}^{-1}dh_1 {h_1}^{-1}dh_1 dh_2 {h_2}^{-1}{h_1}^{-1})  \hspace{7em}$$
 $$ + \frac{1 }{4}  {\rm tr}  ( dh_1 dh_2 {h_2}^{-1}{h_1}^{-1}dh_1 dh_2 {h_2}^{-1}{h_1}^{-1}) $$
$$\hspace{9em} + \frac{1 }{2} {\rm tr}   ( dh_1 dh_2 {h_2}^{-1}dh_2 {h_2}^{-1}dh_2 {h_2}^{-1}{h_1}^{-1}))  $$

$$  C_{3,3} =\frac{-1}{3!} \left( \frac{1}{2 \pi i} \right) ^3 (\frac{1 }{2}{\rm tr} ( dh_1 dh_2 dh_3 {h_3}^{-1}{h_2}^{-1}{h_1}^{-1} ) \hspace{10em}$$
$$  \hspace{10em} -\frac{1 }{2} {\rm tr} (dh_1 h_2 dh_3 {h_3}^{-1}{h_2}^{-1}  dh_2 {h_2}^{-1}{h_1}^{-1})). $$

\end{example}

\section{The Chern-Simons form }
\setcounter{equation}{0}

We briefly recall the notion of the Chern-Simons form in \cite{Chern}.

Let $\pi : E \rightarrow M$ be any principal $G$-bundle and  $\theta$, $\Omega$ denote its connection form and the curvature.
For any $ P \in {\rm  I}^{k} (G) $, we define the $(2k-1)$-form $TP( \theta )$ on $E$ as:\\
$$TP( \theta )  := k \int_{0}^{1} P(\theta \wedge { \phi }_t ^{k-1})dt. $$
Here ${ \phi }_t := t \Omega + \frac{1}{2}t(t-1) [ \theta , \theta ]$.
Then the equation $ d(TP( \theta )) = P( { \Omega }^k ) $ holds and $TP( \theta )$ is called the Chern-Simons form of $P( { \Omega }^k ) $.
When the bundle is flat, its curvature vanishes and hence $ d(TP( \theta )) = P( { \Omega }^k ) = 0 $. 

Now we put the simplicial connection into $TP$ and using the same argument in section 3, then we obtain the Chern-Simons form in $\Omega ^{2p-1}(N \bar{G}) $.
\begin{proposition}
The Chern-Simons form in ${\Omega}^{3}(N \overline{U(n)} )$ which corresponds to the second Chern class $c_2$ is the sum of the following 
$Tc_{0,3}$, $Tc_{1,2}$:
$$
\begin{CD}
0 \\
@AA{d''}A \\
Tc_{0,3} \in {\Omega}^{3} (U(n) )@>{d'}>>{\Omega}^{3} (N \overline{U(n)}(1))\\
@.@AA{d''}A\\
@.Tc_{1,2} \in {\Omega}^{2} (N \overline{U(n)}(1))@>{d'}>> {\Omega}^{2} (N \overline{U(n)}(2))
\end{CD}
$$
$$Tc_{0,3} = \left( \frac{1}{2 \pi i} \right) ^2  \frac{1}{6}{\rm tr}({g^{-1}dg} )^3 \hspace{15em} $$
$$ Tc_{1,2} = \left( \frac{1}{2 \pi i} \right) ^2   \left( \frac{1}{2}{\rm tr}( g_0 ^{-1} dg_0 g_1 ^{-1}dg_1  ) 
-  \frac{1}{2}{\rm tr}( g_0 ^{-1} dg_0 ){\rm tr}(g_1 ^{-1}dg_1  )\right).$$
\end{proposition}
\bigskip
\begin{remark}
The term 
$\left( \frac{1}{2 \pi i} \right) ^2 \frac{1}{2}{\rm tr}( g_0 ^{-1} dg_0 ){\rm tr}(g_1 ^{-1}dg_1  )$ vanishes when we restrict it to $SU(n)$.
\end{remark}

\section{Formulas for a cocycle in a truncated complex}
\setcounter{equation}{0}
In this section, we prove the conjecture due to Brylinski in \cite{Bry}.

At first, we introduce the filtered local simplicial de Rham complex.

\begin{definition}[\cite{Bry}]
The filtered local simplicial de Rham complex $F^p\Omega_{{\rm loc}} ^{*,*}(NG)$ over a simplicial manifold $NG$ is
defined as follows:\\
$$
F^p \Omega^{r,s} _{\rm loc}(NG)=\begin{cases}
 \underrightarrow{\rm lim}_{1 \in V \subset G^r} ~~\Omega ^s(V)  & ~{\rm if}~ s \ge p \\
0  &  {\rm otherwise}. 
\end{cases}
$$
\end{definition}

Let $F^p \Omega^{*}(NG)$ be a filtered complex 
$$
F^p \Omega^{r,s} (NG)=\begin{cases}
 \Omega ^s(NG(r))  & ~{\rm if}~ s \ge p \\
0  &  {\rm otherwise} 
\end{cases}
$$
and $[\sigma_{<p}\Omega ^{*}({NG})]$ a truncated complex
$$
[\sigma_{<p}\Omega ^{r,s}({NG})]=\begin{cases}
0  &  ~{\rm if}~ s \ge p  \\
\Omega ^s(NG(r))  & ~{\rm otherwise}. 
\end{cases}
$$
Then there is an exact sequence: 
$$0 \rightarrow F^p \Omega^{*}(NG) \rightarrow \Omega^{*}(NG) \rightarrow [\sigma_{<p}\Omega^{*}({NG}) ] \rightarrow 0$$
which induces a boundary map 
$$\beta:H^l(NG,[\sigma_{<p}\Omega_{\rm loc} ^{*}]) \rightarrow H^{l+1}(NG,[F^p \Omega^{*} _{\rm loc}]).$$

Let $\omega_1 + \cdots + \omega_p$, $\omega_{p-q} \in \Omega^{p+q}(NG(p-q))$
be the cocycle in $\Omega^{2p}(NG)$ which represents the $p$-th Chern character. By using this cocycle, Brylinski constructed a cochain $\eta$ in $[\sigma_{<p}\Omega ^{*} _{\rm loc}({NG})]$ in the following way.

We take a contractible open set $U \subset G$ containing $1$. Using the same argument in \cite[Lemma 9.7]{Dup2}, we can 
construct mappings $\{ \sigma_l:{ \Delta }^l\times  U^l \rightarrow U \}_{0 \le l}$ inductively with the following properties:\\
(1)~$\sigma_0(pt)=1$;\\
(2)
$$
{\sigma}_l({\varepsilon}^{j}(t_0, \cdots , t_{l-1});h_1, \cdots, h_l)=\begin{cases}
{\sigma}_{l-1}( t_0, \cdots, t_{l-1};{\varepsilon}_{j}(h_1, \cdots , h_l))   &  ~{\rm if}~j \ge 1 \\
h_1 \cdot {\sigma}_{l-1}(t_0, \cdots, t_{l-1};h_2, \cdots , h_l)   &  ~{\rm if}~j=0.
\end{cases}
$$
Then we define mappings $\{f_{m,q}: { \Delta }^q \times  U^{m+q-1} \rightarrow G^m \}$ by 
$$f_{m,q}(t_0, \cdots , t_q;h_1, \cdots, h_{m+q-1}):=(h_1, \cdots , h_{m-1}, {\sigma}_q(t_0, \cdots , t_q;h_m, \cdots , h_{m+q-1})).$$
We can check $f_{m,q} \circ {\varepsilon}_{j} = f_{m,q+1} \circ {\varepsilon}^{j-m+1}:{ \Delta }^q \times  U^{m+q} \rightarrow G^m $ 
if $m \le j \le m+q$ and $f_{m,q} \circ {\varepsilon}_{j} = {\varepsilon}_{j} \circ f_{m+1,q}$ if  $m-1 \ge j \ge 0$
and ${\varepsilon}_{m} \circ f_{m+1,q}  = f_{m,q+1}
\circ {\varepsilon}^0$ holds.

We define a $(2p-m-q)$-form $\beta_{m,q}$ on $U^{m+q-1}$ by $\beta_{m,q}=(-1)^m \int_{{ \Delta }^q}f_{m,q} ^* {\omega}_m$.
Then the cochain $\eta$ is defined as the sum of following $\eta_l$ on $U^{2p-1-l}$ for $0 \le l \le p-1$:
$$\eta _l:=\sum_{m+q=2p-l,~ m \ge 1}\beta _{m,q}.$$

Now we are ready to state the theorem whose statement is conjectured by Brylinski \cite{Bry}.

\begin{theorem}
$\eta := \eta _0+ \cdots + \eta_{p-1}$ is a cocycle in $[\sigma_{<p}\Omega_{\rm loc} ^{*}(NG)]$ whose cohomology class is mapped 
to $[\omega_1 + \cdots + \omega_p]$ in $H^{2p}(NG,[F^p \Omega^{*} _{\rm loc}])$ by a boundary map 
$\beta : H^{2p-1}(NG,[\sigma_{<p}\Omega_{\rm loc} ^{*}]) \rightarrow H^{2p}(NG,[F^p \Omega^{*} _{\rm loc}])$.
\end{theorem}
\begin{proof}
To prove this, it suffices to show the equation below holds true for any $l$ which satisfies  $0 \le l \le 2p-1$ since $\omega_{2p-l}=0$ if $0 \le l \le p-1$:
$$\sum_{i=0} ^{2p-l}(-1)^i {\varepsilon}_{i} ^{*} \eta_l = (-1)^{2p-l+1}d \eta_{l-1} + \omega_{2p-l}.$$
The left side of this equation is equal to 
$$\sum_{m+q=2p-l,~ m \ge 1} (-1)^m \left( \int_{{ \Delta }^q} \sum_{i=0} ^{m-1}  (-1)^i(f_{m,q} \circ  {\varepsilon}_{i} )^* {\omega}_m  +  \int_{{ \Delta }^q} \sum_{i=m} ^{m+q} (-1)^i(f_{m,q} \circ  {\varepsilon}_{i} )^* {\omega}_m  \right).$$
We can check that 
$$\sum_{i=0} ^{m-1}(-1)^i (f_{m,q} \circ  {\varepsilon}_{i} )^* {\omega}_m = f_{m+1,q} ^* ( \sum_{i=0} ^{m-1} (-1)^i {\varepsilon}_{i} ^*  {\omega}_m)$$
hence by using the cocycle relation $\sum_{i=0} ^{m+1} (-1)^i {\varepsilon}_{i} ^*  {\omega}_m = (-1)^m d \omega_{m+1}$, we can see the following holds:

$$\int_{{ \Delta }^q} \sum_{i=0} ^{m-1} (-1)^i(f_{m,q} \circ  {\varepsilon}_{i} )^* {\omega}_m  
= \int_{{ \Delta }^q}(-1)^m d f_{m+1,q} ^* \omega _{m+1} \hspace{10em}$$
$$\hspace{5em} - \left( (-1)^m  \int_{{ \Delta }^q}( {\varepsilon}_{m}  \circ f_{m+1,q})^* {\omega}_m 
+ (-1)^{m+1}  \int_{{ \Delta }^q}( {\varepsilon}_{m+1}  \circ f_{m+1,q})^* {\omega}_m\right).$$

Note that $ \int_{{ \Delta }^q}( {\varepsilon}_{m+1}  \circ f_{m+1,q})^* {\omega}_m = 0 $ for $q \ge 1$ and
$\int_{{ \Delta }^q}( {\varepsilon}_{m+1}  \circ f_{m+1,q})^* {\omega}_m = \omega_{2p-l}$ if $q=0$.

We can also check that 
$$\int_{{ \Delta }^q} \sum_{i=m} ^{m+q} (-1)^i (f_{m,q} \circ  {\varepsilon}_{i} )^* {\omega}_m 
=  \int_{{ \Delta }^q} \sum_{i=m} ^{m+q} (-1)^i(f_{m,q+1} \circ  {\varepsilon}^{i-m+1} )^* {\omega}_m.  $$
We set $j=i-m+1$, then we see that $\int_{{ \Delta }^q} \sum_{i=m} ^{m+q} (-1)^i(f_{m,q+1} \circ  {\varepsilon}^{i-m+1} )^* {\omega}_m$
is equal to
 $$ 
\sum_{j=0} ^{q+1} \left( (-1)^{j+m-1} \int_{{ \Delta }^q}(f_{m,q+1} \circ  {\varepsilon}^{j} )^* {\omega}_m \right) - 
(-1)^{m-1} \int_{{ \Delta }^q}({\varepsilon}_{m} \circ f_{m+1,q})^* {\omega}_m$$
since ${\varepsilon}_{m} \circ f_{m+1,q}  = f_{m,q+1} \circ {\varepsilon}^0$.

From above, we can see that $\sum_{i=0} ^{2p-l}(-1)^i {\varepsilon}_{i} ^{*} \eta_l$ is equal to
$$\omega_{2p-l}+ \sum_{m+q=2p-l,~ m \ge 1} \left(  \int_{{ \Delta }^q}d f_{m+1,q} ^* \omega _{m+1}  + \sum_{j=0} ^{q+1} (-1)^{j-1} \int_{{ \Delta }^q}(f_{m,q+1} \circ  {\varepsilon}^{j} )^* {\omega}_m \right). $$
On the other hand, for any $(m',q')$ which satisfies $m'+q'=2p-(l-1)$ the following equation holds:
$$(-1)^{q'}d\int_{{ \Delta }^{q'}} f_{m',q'} ^* \omega _{m'}= \int_{{ \Delta }^{q'}} df_{m',q'} ^* \omega _{m'}
-\sum_{j=0} ^{q'} \int_{{ \Delta }^{q'-1}} (-1)^j {{\varepsilon}^{j}} ^* f_{m',q'} ^* \omega _{m'}.$$
Therefore $(-1)^{2p-l+1}d \eta_{l-1}$ is equal to 
$$ \sum_{m'+q'=2p-l+1,~ m' \ge 1}\left(\int_{{ \Delta }^{q'}} df_{m',q'} ^* \omega _{m'}
-\sum_{j=0} ^{q'} \int_{{ \Delta }^{q'-1}} (-1)^j {{\varepsilon}^{j}} ^* f_{m',q'} ^* \omega _{m'}\right).$$
This completes the proof.

\end{proof}

\begin{remark}
Let me explain Brylinski's motivation in \cite{Bry} to introduce these complexes 
and the conjecture briefly. Let $LU$ be the free loop group of a contractible open set $U \subset G$ containing $1$ and ${\rm ev}:LU \times  S^1 \to U$ be the evaluation map, i.e. for $\gamma \in LU$ and $\theta \in S^1$, ${\rm ev}(\gamma , \theta)$ is defined as $\gamma (\theta)$. Then $\int _{S^1} {\rm ev}^*$ maps $\eta_{1} \in \Omega ^1(U^{2p-2})$ to a cochain in  $\Omega ^0(LU^{2p-2})$.
This cochain defines a cohomology class in local cohomology group $H^{2p-2} _{\rm loc}(LU, {\mathbb C})$. Brylinski constructed a natural map from
$H^{2p-2} _{\rm loc}(LU, {\mathbb C})$ to the the Lie algebra cohomology $H^{2p-2}(L \mathcal{G}, {\mathbb C})$. Then as a special case
$p=2$, he used the cocycle in the local truncated complex $[\sigma_{<2}\Omega^{3} _{\rm loc}(NG)]$
to construct the standard Kac-Moody $2$-cocycle. 
He treated not only the free loop group but also the gauge group ${\rm Map}(X,G)$ for a compact oriented manifold $X$.
\end{remark}
{\bf Acknowledgments.} \\
I am indebted to Professor H. Moriyoshi for helpful discussion and good advice. I would like to
thank the referee for his/her several suggestions to improve this paper.


Graduate School of Mathematics, Nagoya University, Furo-cho, Chikusa-ku, Nagoya-shi, Aichi-ken, 464-8602, Japan. \\
e-mail: suzuki.naoya@c.mbox.nagoya-u.ac.jp
\end{document}